\documentclass[a4paper]{amsart}
\usepackage[final,theoremdefs]{allergy}

\usepackage{amsaddr}

\usepackage{enumitem}

\usepackage{tikz}
\usetikzlibrary{decorations.pathreplacing,shapes,arrows,positioning}
\tikzset{
	vertex/.style = {
		circle,
		fill            = black,
		outer sep = 2pt,
		inner sep = 1pt,
	}
}
\usetikzlibrary{spy}
\usetikzlibrary{matrix}

\usepackage{algpseudocode}
\usepackage{algorithm}

\usepackage{siunitx} 

\usepackage{pdfsync}
\graphicspath{{fig/}}

\usepackage[labelformat=simple]{subcaption}

\usepackage{pdfsync}
\usepackage{bbm}

\usepackage{enumitem}
\usepackage[T1]{fontenc} 

\newcommand*\RR{\mathbb{R}}

\newcommand*\NN{\mathbb{N}}

\newcommand {\MM} { {\mathcal{M}} }

\newcommand*\one{\boldsymbol{1}}

\newcommand*\loss{\ell}

\renewcommand{\geq}{\geqslant}
\renewcommand{\leq}{\leqslant}

\renewcommand{\geq}{\geqslant}
\renewcommand{\leq}{\leqslant}

\newcommand {\Chi} {{\bf \raise 2pt \hbox{$\chi$}} }

\newcommand*\Cont{\mathcal{C}}
\newcommand*\ndet{m}

\newcommand{\scl}[2]{\pairing*{#1}{#2}}

\newcommand*\dd{\mathrm{d}}

\newcommand{\beq}{\begin{equation}}
\newcommand{\eeq}{\end{equation}}
\newcommand{\bea} {\begin{array}{rl}}
\newcommand{\eea} {\end{array}}
\newcommand{\bepa}{\left\{ \begin{array}{l}}
\newcommand{\eepa} {\end{array}\right.}
\newcommand{\bmu}{\begin{multline}}
\newcommand{\emu}{\end{multline}}
\DeclareMathOperator*{\argmin}{arg\,min}

\title{Statistical model and ML-EM algorithm for emission tomography with known movement}

\date{\today}
\author{Camille Pouchol}
\address{MAP5 Laboratory, FP2M, CNRS FR 2036, Universit\'e de Paris, 75006 Paris, France.}

\author{Olivier Verdier}
\address{Department of Computing, Electrical Engineering and Mathematical Sciences, Western Norway University of Applied Sciences, Bergen, Norway.}

\begin{document}

\newcounter{assum}

\begin{abstract}
  In Positron Emission Tomography (PET), movement leads to blurry reconstructions when not accounted for.
  Whether known a priori or estimated jointly to reconstruction, motion models are increasingly defined in continuum rather that in discrete, for example by means of diffeomorphisms.
  The present work provides both a statistical and functional analytic framework suitable for handling such models.
  It is based on time-space Poisson point processes as well as regarding images as measures, and allows to compute the maximum likelihood problem for line-of-response data with a known movement model. 
  Solving the resulting optimisation problem, we derive an Maximum Likelihood Expectation Maximisation (ML-EM) type algorithm which recovers the classical ML-EM algorithm as a particular case for a static phantom.
  The algorithm is proved to be monotone and convergent in the low-noise regime.
  Simulations confirm that it correctly removes the blur that would have occurred if movement were neglected.

\end{abstract}

\maketitle

\section{Introduction}

In Positron Emission Tomography (PET), line-of-response data consists of the times of simultaneous detections of two photons, in each of the $\ndet$ pairs of detectors, from which one aims at reconstructing the underlying image $\mu \in X$, for some suitable space of images $X$.
When the phantom is static, the times are grouped into numbers of detections $y_i$ per detector $i \in \set{1, \ldots, \ndet}$.
A good statistical model for the problem is then $y = \mathcal{P}(A \mu)$, {i.e.}, the data is obtained as $\ndet$ independent Poisson random variables $y_i$ of mean $(A \mu)_i$ where $A \colon X \to \mathbb{R}^\ndet$ is a known operator modelling the scanner geometry. 

This inverse problem is in practice solved through variants of the iterative Maximum Likelihood Expectation Maximisation (ML-EM) algorithm \[\mu_{k+1} = \frac{\mu_k}{A^T 1} A^T\Big(\frac{y}{A \mu_k}\Big),\] which is aimed at maximising the likelihood associated to the above statistical model, {i.e.}, at minimising $\loss(\mu) := d(y || A \mu)$ over $\mu \in X$, $\mu \geq 0$, where $d$ is the Kullback--Leibler divergence~\cite{Shepp1982, Vardi1985, Verdier2020}. 

Reconstruction methods in medical imaging suffer from blurring effects if the phantom moves during acquisition time, unless movement is taken into account in the reconstruction process.
Cardiac or thoracic PET scans are typical instances of this problem.

\emph{Motion estimation} refers to methods which take movement into account by estimating it.
This can be done either prior to reconstruction, or jointly with it, see~\cite{Gigengack2015, Dawood2008, Rahmim2013, Reader2014} for a review.
A first class of methods is based on a discrete parameterisation of movement, such as~\cite{Jacobson2003, Jacobson2006}.
Some approaches instead rely on the situation where a physical device allows to group counts per phase (called \textit{gates}) in which the movement can be assumed to be stationary. For instance~\cite{chan18_non_rigid_event_by_event} performs motion correction in each gate. 
Hence, the model is discrete in space, and also in time (because of the gating), although motion estimation is done within the gates.
A related approach is \cite{qiao06_motion_incor_recon_method_gated_pet_studies}, in which the motion is measured with a CT scanner during the acquisition.

However, models with a continuous description of movement, typically by means of diffeomorphisms, are gaining popularity in the context of PET~\cite{Blume2010, Hinkle2012}, and more broadly in imaging sciences~\cite{Younes2010}. 
Considering $X$ as a functional space of functions defined over a compact $K \subset \mathbb{R}^p$, these approaches assume that the activity
is modified by operators $\mathcal W_t \colon X \rightarrow X$.

If we assume as is common that \(\mathcal{W}_t\) is defined via a diffeomorphism for any \(t\in[0,1]\),
as diffeomorphisms do not preserve grids, the latter approach does not easily lend itself to discretisation (in which case $X$ becomes a finite-dimensional space). For instance, both \cite{qiao06_motion_incor_recon_method_gated_pet_studies} and \cite{chan18_non_rigid_event_by_event} have to explicitly resort to interpolation.
We also note that none of the papers \cite{chan18_non_rigid_event_by_event,Jacobson2003, qiao06_motion_incor_recon_method_gated_pet_studies} carry out a theoretical study as ours.

\subsection*{Contributions.}
The aim of the present paper is twofold.
\begin{itemize}
\item First, building up on~\cite{Verdier2020}, we introduce a continuous mathematical framework which incorporates any movement model, and whereby the maximum likelihood problem associated to the times of detections may be derived.
\item Second, based on this construction, we propose and analyse an ML-EM type algorithm for the maximum likelihood problem associated to a known continuous movement model. 
\end{itemize}

When the movement is in fact static, we recover the classical ML-EM algorithm as a particular case.
Even in the static case, optimal solutions to the maximum likelihood problem can be singular measures (sums of point masses), and it is the standard outcome in the noisy case~\cite{Mair1996,Verdier2020}.
We thus use measures to model images in this continuous context.

We also emphasise that in the case of \emph{gated} data, an ML-EM algorithm has been derived informally in the literature~\cite{qiao06_motion_incor_recon_method_gated_pet_studies, Hinkle2012, Li2019, Oktem2019}.
We recover that algorithm as well when the $\mathcal W_t$ are assumed to be piecewise constant in time.

The proposed approach applies to cases where the movement model is known.
This can be for scans of phantoms where movement is controlled, or when the motion model is estimated prior to reconstruction.
An example is provided in \autoref{translation_evolution} and \autoref{translation_comp}; see \autoref{sec:translation} for more details.

\begin{figure}
  \centering
\captionsetup{size=small}
\begin{subfigure}{0.32\textwidth}
\includegraphics[width=\linewidth]{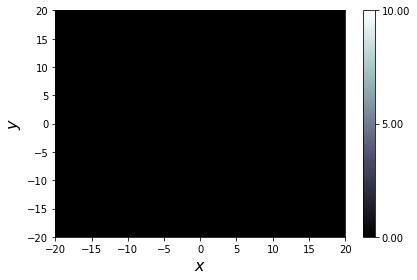} 
\caption{$t = 0$} 
\end{subfigure} 
\hfill
\begin{subfigure}{0.32\textwidth}
\includegraphics[width=\linewidth]{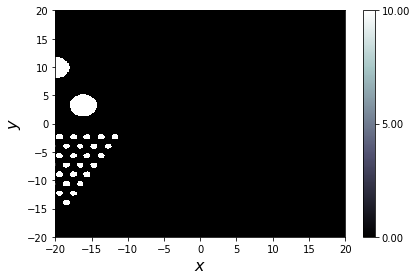} 
\caption{$t = 0.2$} 
\end{subfigure} 
\hfill
\begin{subfigure}{0.32\textwidth}
\includegraphics[width=\linewidth]{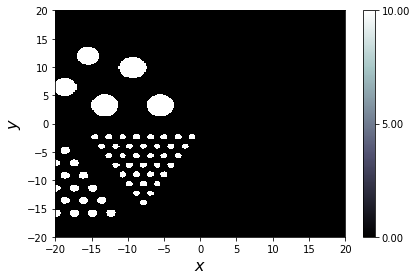} 
\caption{$t = 0.4$}
\end{subfigure}

\bigskip

\begin{subfigure}{0.32\textwidth}
\includegraphics[width=\linewidth]{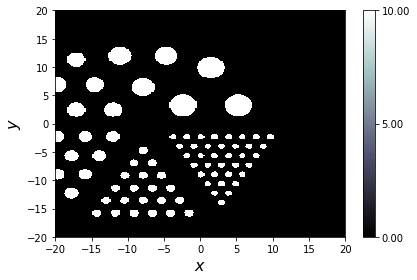} 
\caption{$t = 0.6$} 
\end{subfigure} 
\hfill
\begin{subfigure}{0.32\textwidth}
\includegraphics[width=\linewidth]{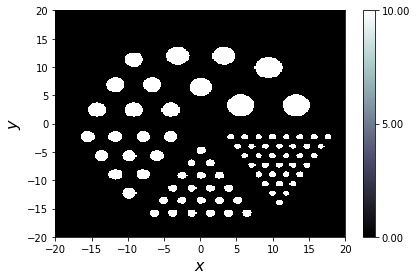} 
\caption{$t = 0.8$} 
\end{subfigure} 
\hfill
\begin{subfigure}{0.32\textwidth}
\includegraphics[width=\linewidth]{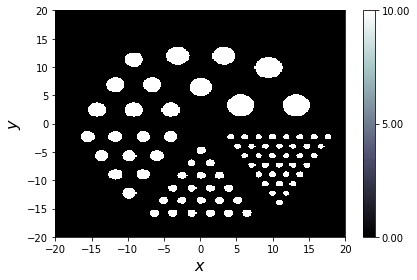} 
\caption{$t = 1$} 
\end{subfigure} 

\caption{Evolution of a template evolution with translation operators.} \label{translation_evolution}
\end{figure}

\begin{figure}

\begin{subfigure}{0.49\textwidth}
\includegraphics[width=\linewidth]{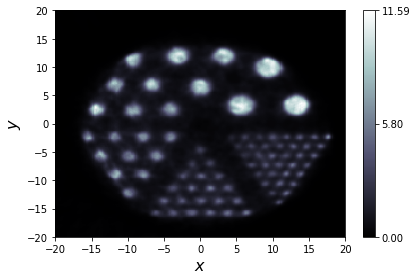} 
\caption{Proposed ML-EM algorithm} \label{trans:proposed}
\end{subfigure} 
\hfill
\begin{subfigure}{0.49\textwidth}
\includegraphics[width=\linewidth]{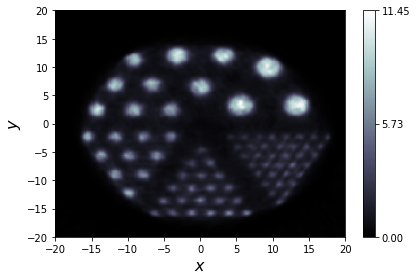} 
\caption{Classical ML-EM on static data} \label{trans:classical}
\end{subfigure} 

\bigskip
\begin{subfigure}{0.49\textwidth}
\includegraphics[width=\linewidth]{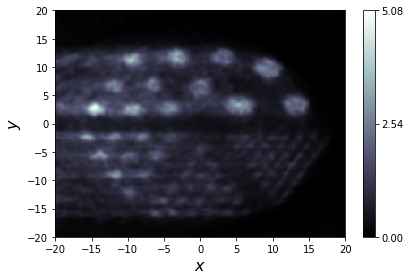} 
\caption{Classical ML-EM on full aggregated data  \quad} \label{trans:aggregated}
\end{subfigure} 
\hfill
\begin{subfigure}{0.49\textwidth}
\includegraphics[width=\linewidth]{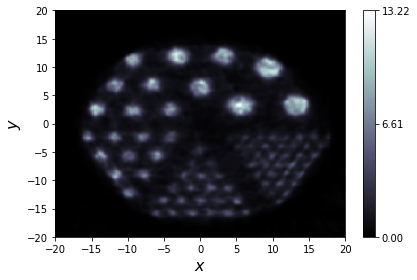} 
\caption{Classical ML-EM on partial aggregated data on $[\frac{3}{4},1]$, scaled} \label{trans:partial}
\end{subfigure} 
\caption{$10$th iterate of the proposed algorithm \eqref{mlem} compared to the $10$th iterate of the classical ML-EM algorithm in various cases.} \label{translation_comp}
\end{figure}



Our framework can also serve as a building block for joint motion estimation and reconstruction, when the movement model is unknown.
A typical strategy to solve the resulting optimisation problem is use alternating minimisation, successively  estimating the image and the transformation (typically a diffeomorphism)~\cite{Hinkle2012, Li2019, Oktem2019}.
Thus, given a current estimate of the transformations $\mathcal W_t$, 
the next iterate for the image is obtained by maximising
the corresponding likelihood,
which is
exactly
what our algorithm does.

\subsection*{Outline of the paper}
The paper is organised as follows.
In \autoref{sec2}, we introduce the notations as well as the modelling through time-space Poisson point processes.
We derive the likelihood associated to the times of detections for this statistical model.
Finally, we provide the corresponding optimality conditions.
The ML-EM algorithm is introduced in \autoref{sec3}, where we prove its monotonicity and analyse its convergence properties.
In \autoref{sec4}, numerical simulations are presented, showing that our algorithm performs like the classical ML-EM, and avoids blurring effects when compared to neglecting motion. 
In \autoref{sec5}, we conclude by discussing the reach and limits of our approach, together with some perspectives.

\section{Maximum likelihood function}
\label{sec2}
\subsection{Notations} 

For a given compact $E \subset \mathbb{R}^d$, we denote $\MM(E)$ the set of Radon measures defined on $E$, {i.e.}, the topological dual of continuous functions $\mathcal{C}(E)$ over~$E$. Endowing $\MM(E)$ with its weak-$\ast$ topology, its dual is given by $\mathcal{C}(E)$.
We denote by \(\scl{\mu}{f}\) the pairing of a measure $\mu \in \MM_+(E)$ and a function $f \in \mathcal{C}(E)$. 

The set of nonnegative measures will be denoted $\MM_+(E)$.
By the Riesz--Markov representation theorem, we may also regard a measure $\mu \in \MM_+(E)$ as a Borel measure, and we will sometimes write $\mu(B)$ for the measure of a measurable set $B \subset E$.

For $\mu, \nu$ two measures in $\MM(E)$, the notation $\mu \ll \nu$ means that $\mu$ is absolutely continuous with respect to $\nu$.

An operator $\mathcal{W} \colon \mathcal{M}(E) \rightarrow \mathcal{M}(E)$ {i.e.}, a linear and continuous mapping in the weak-$\ast$ topology, is defined through its adjoint $\mathcal{W}^\ast \colon \mathcal{C}(E) \rightarrow \mathcal{C}(E)$ by
\[\scl{\mathcal{W} \mu}{f} :=\scl{\mu}{ \mathcal{W}^\ast f}\] for any $f \in \mathcal{C}(E)$. 

Note that the adjoint of an operator of $\mathcal{M}(E)$ is well-defined as a mapping from $\mathcal{C}(E)$ onto itself, meaning that such an operator indeed may be defined through its adjoint, see~\cite[Proposition 3.14]{Brezis2010}.

\subsection{Modelling} 
We quickly recall the physics underlying Positron Emission Tomography. A radiotracer injected into the patient diffuses into tissues and then disintegrates by emitting positrons. A given positron, after a very short travel distance, interacts with an electron, yielding two photons emitted in (uniformly) random opposite directions. Such photons are then detected simultaneously by a pair of detectors.

We are given a compact $K \subset \mathbb{R}^p$ (having $p=2$, $p=3$ in mind for the application) on which the image is defined. We assume that the emission process is defined by a time-space Poisson point process. The intensity of the latter is a measure in $\mathcal{M}_+([0,1] \times K)$, where, without loss of generality,  we fix the final time to one.

We assume that the intensity at time \(t\)
is given by a known linear transformation of the unknown measure $\mu \in \MM_+(K)$.
In other words,
the activity at time \(t\) is the measure $\mathcal{W}_t  \, \mu$,
where the linear operators
\[\mathcal{W}_t \colon \mathcal{M}(K) \rightarrow \mathcal{M}(K), \qquad t \in [0,1]
  ,
\]
are known.

Hence, we define the measure underlying the time-space Poisson process by
\[ [t_1, t_2] \times B \longmapsto \int_{t_1}^{t_2} (\mathcal{W}_t \mu)(B) \, dt, \qquad 0 \leq t_1 < t_2\leq 1,\; B \subset K \text{ Borel set},\]
which we denote $\mathcal{W}_t \mu$ with a slight abuse of notation.
This measure is well-defined under the assumption~\eqref{Integrable}, see the full set of assumptions in the next subsection.

Independently of the emission process associated, a point emitted at $x \in K$ and time $t \in [0,1]$ then has a relative probability $a_i(x)$ to be detected in detector \(i\), and we thus assume 
\begin{equation}
  \label{eq:posdet}
a_i \in \mathcal{C}(K), \qquad a_i \geq 0, \qquad i = 1, \ldots, \ndet.
\end{equation}
We assume that the detection also occurs at time $t$, which is an excellent approximation.


We now define the PET operator $A \colon \MM(K) \rightarrow \mathbb{R}^\ndet$ by \[(A \mu)_i :=   \scl{\mu}{a_i} = \int_K a_i \, \dd\mu , \quad i= 1,\ldots, \ndet.\]
Note that the adjoint $A^*: \mathbb{R}^\ndet \rightarrow \mathcal{C}(K)$ of $A$ is given by
\begin{equation*}
 A^* \lambda = \sum_{i=1}^\ndet \lambda_i a_i, \qquad \lambda \in \ \mathbb{R}^\ndet.
\end{equation*}

\subsection{Assumptions on the Transformations}

We make the following assumptions on the transformations \(\mathcal{W}_t\), for \(t\in[0,1]\).
\begin{description}
\item[Nonnegativity Preservation] 
  \begin{equation}
    \label{Preservation}
    \mu \geq 0 \implies \mathcal{W}_t \mu \geq 0, \qquad
    t \in [0,1]
  \end{equation}
  \item[Integrability]
   denoting $1$ for the constant function in $\mathcal{C}(K)$, we assume
    \begin{equation}
      \label{Integrable}
      \int_{[0,1] \times K} \mathcal{W}_t^* 1 \; < +\infty.
    \end{equation}
\end{description}
Denoting $\boldsymbol{1} = (1, \ldots, 1) \in \mathbb{R}^\ndet$, the condition \eqref{Integrable} allows us to define the following integral:
\begin{equation}
  f \coloneqq \int_0^1 \mathcal{W}_t^* A^* \one\, \dd t
\end{equation}
A consequence of \eqref{Preservation} and \eqref{eq:posdet} is that \(f \geq 0\).
We make the further blanket assumption that \(f \neq 0\), that is (since \(K\) is compact)
\begin{equation}
  \label{Positive}
  \exists c > 0,\qquad
  f(x)  \geq c 
  \qquad
  x \in K
  .
\end{equation}


The assumption~\eqref{Preservation} equivalently writes \(g \geq 0 \implies \mathcal{W}_t^\ast g \geq 0\) for any $t \in [0,1]$, $g \in \mathcal{C}(K)$, $g\geq 0$. It means that a transformation \(\mathcal{W}_t\) cannot create illegal (negative) activity.
The assumption~\eqref{Integrable} essentially ensures that there is a finite activity over time.
The blanket assumption~\eqref{Positive}
is natural:
if \(f(x)= 0\), then \(\mathcal{W}_t^{\star} a_i(x) = 0\) for all detectors~\(i\) and times \(t\), which means that nothing can ever be measured at the point $x$, and it makes no sense to try and estimate the image there. 

\subsection{Maximum likelihood}

 We assume that there are $n_i$ detections in a given detector $i \in \set{1, \dotsc, \ndet}$, detected at times $t_j^i$, $j=1, \ldots, n_i$. We do not require that the times be ordered. We also denote \[n := \sum_{i=1}^\ndet n_i,\] the total number of detections.
 
 Our aim is to derive the likelihood for the problem of estimating $\mu$ from the data given by the number of points and times of detections, namely
 \[n_i, \quad t_j^i, \quad j = 1, \ldots, n_i, \; i = 1, \ldots, \ndet.\]
 
We introduce a couple of additional notations. First, we define for a time $t \in [0,1]$ and the measure $\mu \in \MM_+(K)$ the function 
\[
  \beta_i(t) \coloneqq (A \mathcal{W}_t \mu)_i, \qquad 
  i = 1, \ldots, \ndet
  ,
\]
or equivalently, \(\beta_i(t) = \scl{\mathcal W_t \mu}{a_i}\).
We also define a collection \(\Gamma\) of $n$ continuous functions by
\[
  \Gamma :=  \setc{\mathcal{W}_{t_j^i}^\ast a_i}{i = 1, \ldots, \ndet, \; j = 1, \ldots, n_i}
  .
\]
We first identify the point process from which the data is drawn. 
\begin{proposition}
\label{Inhomogeneous}
For $i \in \set{1, \dotsc, \ndet}$, the number of points $n_i$ and times $t_j^i$ are drawn from independent inhomogeneous (in time) Poisson processes on $[0,1]$, with respective intensities given by the functions $\beta_i$.
\end{proposition}

\begin{proof}
  Since the emission process is independent from that of the detection process, we use the thinning property (\cite[Theorem 5.8]{Last2017}) to assert that for each $i \in \set{1, \dotsc, \ndet}$  the point process defined by the points detected by detector $i$ is also a time-space Poisson process, with underlying measure
\[a_i \, \mathcal{W}_t \mu \in \MM_+([0,1] \times K),\] 
{i.e.}, the measure  \[ [t_1, t_2] \times B \longmapsto \int_{t_1}^{t_2} (a_i \mathcal{W}_t \mu)(B) \, dt, \qquad 0 \leq t_1 < t_2\leq 1,\;  B \subset K \text{ Borel set}.\]
Furthermore and still by the thinning property, all these processes are independent.

For a given $i \in \set{1, \ldots, \ndet}$, the number of points $n_i$ and the times of detections $t_j^i$ are that of a (time) Poisson process defined on $[0,1]$, since it corresponds
to marginalising with respect to $x \in K$ the associated time-space Poisson process. The resulting object is thus an inhomogeneous (in time) Poisson process, with intensity obtained by integrating the measure $a_i \, \mathcal{W}_t \mu$ in space:
\[\forall t \in [0,1],\quad
  \int_{K}\dd\paren{a_i\mathcal{W}_t \mu} = \pairing{\mathcal{W}_t \mu}{a_i} = \beta_i(t)
  .
\]
\end{proof}

\begin{remark}
~\autoref{Inhomogeneous} also provides a way to simulate synthetic data according to the above model, see~\autoref{sec4} about numerical simulations for more details on how this may be done.
\end{remark}

We are now in a position to derive the maximum likelihood problem, namely:
\begin{corollary}
  The maximum likelihood problem is given by
  \[\operatorname*{min}_{\mu \in \MM_+(K)} \loss(\mu)
    ,
  \] where
  \[\loss(\mu):= \scl{\mu}{f} -
    \sum_{\gamma\in\Gamma} \log(\scl{\mu}{\gamma})
    ,
  \]
with the convention that $\loss(\mu) = +\infty$ for $\mu$ not in
\[
  \operatorname{dom}(\ell):=\setc[\big]{\mu \in \MM_+(K)}{ \scl{\mu}{\gamma} > 0, \; \gamma\in\Gamma}
  .
\]
\end{corollary}

\begin{proof}
The negative log-likelihood for the number of points $n$ and times $t_j$ for an inhomogeneous Poisson point process of intensity $\beta \in L^1(0,1)$ is known to be given up to constants by 
\[ \int_0^1 \beta(t) \,dt - \sum_{j=1}^{n} \log\big(\beta (t_j)\big).\]
 The integrability condition~\eqref{Integrable} together with the preservation of nonnegativity~\eqref{Preservation} implies that  $t \mapsto \mathcal{W}_t^\ast a_i \in L^1(0,1; \mathcal C(K))$, and in particular the intensity functions $\beta_i$ defined in the above proof all lie in $L^1(0,1)$.
 
By independence, summing the above negative log-likelihood over $i \in \set{1, \ldots, \ndet}$ we find the full negative log-likelihood
\begin{align}
\loss(\mu) & = \sum_{i=1}^\ndet \int_0^1 \beta_i(t) \,dt -\sum_{i=1}^\ndet \sum_{j=1}^{n_i} \log\big(\beta_i (t_j^i)\big),\\
  & = \pairing{\mu}{f} - \sum_{\gamma\in\Gamma} \log(\scl{\mu}{\gamma})
    .
\end{align}
The last line just uses the fact that
$\beta_i (t_j^i) = \pairing{\mu}{W_{t_j^i}^* a_i}$
whereas the exchange of order of integration in the first term can be performed by Fubini's theorem, owing to~\eqref{Integrable}.
\end{proof}

\begin{remark}
In the stationary case, {i.e.}, $\mathcal{W}_t = \mathrm{Id}$ over $[0,1]$, we have 
$f = A^\ast \one$ and for $\gamma \in \Gamma$ defined by $i \in \set{1, \ldots, \ndet}$, $j  \in \set{1, \ldots, n_i}$, $\gamma = \mathcal{W}_{t_j^i}^\ast a_i= a_i$. 
Hence \[\sum_{\gamma\in\Gamma} \log(\scl{\mu}{\gamma}) = \sum_{i=1}^\ndet \sum_{j=1}^{n_i} \log(\scl{\mu}{a_i}) = \sum_{i=1}^\ndet n_i \log\paren[\big]{(A\mu)_i},\] so that the loss function writes
\[
  \loss(\mu) =\scl{\mu}{A^\ast \one} - \sum_{i=1}^\ndet n_i \log\paren[\big]{(A\mu)_i}
  .
\]
Up to constants, this is nothing but the usual negative log-likelihood $d(y || A \mu)$ with $y = (n_1, \ldots, n_\ndet)$ used in the static case~\cite{Verdier2020}, and $d$ the Kullback-Leibler divergence defined for nonnegative vectors $u, v \in \mathbb{R}^\ndet$ by \[d(u||v) = \sum_{i=1}^\ndet  u_i \log\Big(\frac{u_i}{v_i}\Big) - u_i + v_i,\]
with value $+\infty$ if there exists $i \in \{1,\ldots, \ndet\}$ such that $u_i >0$, $v_i  = 0$.
\end{remark}

%
%
%
%

\subsection{Optimality conditions.}

We derive below the optimality conditions for the maximum likelihood problem.

In order to do so, we endow $\MM(K)$ with its strong topology. The function $\loss$ takes finite values and is differentiable on its domain $\operatorname{dom}(\loss)$, which is open. 
For $\mu \in \operatorname{dom}(\loss)$, we readily compute
\begin{align}
  \label{eq:nablaloss}
  \nabla \loss(\mu)  = f -\sum_{\gamma\in\Gamma}\frac{\gamma} {\scl{\mu}{{\gamma}}},
\end{align}
an element of $\Cont(K)$.
\begin{proposition}
A measure $\mu^\star \in \operatorname{dom}(\loss)$ is optimal if and only if 
\begin{align}
\label{kkt}
  \nabla \loss(\mu^\star )  &\geq 0 \; \text{on} \; K,  \\
  \nabla \loss(\mu^\star )  &=  0 \; \text{on} \; \operatorname{supp}(\mu^\star ).
\end{align}
\end{proposition}
\begin{proof}
Since the function $\loss$ is convex, a measure $\mu^\star \in \operatorname{dom}(\loss)$ is optimal if and only if 
\[ \nabla \loss(\mu^\star) \in - N_{\MM_+(K)}(\mu^\star).\]
where \[N_{\MM_+(K)}(\mu):= \setc*{g \in \Cont(K)}{\forall \nu \in \MM_+(K), \, \pairing{\nu - \mu}{g}\leq 0}\] is the normal cone to $\MM_+(K)$ at $\mu$.
From~\cite[Lemma 3.5]{Verdier2020}, the normal cone can be characterised by
\[
  N_{\MM_+(K)}(\mu) = \setc[\big]{g \in \Cont(K)}{ g \leq 0 \text{ on } K,\;\;
    g = 0 \text{ on } \operatorname*{supp}(\mu)}
  ,
\]
   and the claim follows. 
\end{proof}
Note that this optimality criterion shows that if there exists a measure $\mu$ such that \[ \sum_{\gamma\in\Gamma}\frac{\gamma} {\scl{\mu}{{\gamma}}} = f,\] then $\mu$ is optimal.

\begin{corollary}
The infimum of $\loss$ is a minimum.
\end{corollary}
\begin{proof}
  If an optimum exists,
  since $\nabla \loss(\mu^\star) = 0$ vanishes on the support of $\mu^\star$,
  we obtain \(\pairing{\mu^{\star}}{\nabla\loss(\mu^{\star})} = 0\).
  A computation using \eqref{eq:nablaloss} shows on the other hand that \[\pairing{\mu^{\star}}{\nabla\loss(\mu^{\star})} = \pairing{\mu^{\star}}{f} - \#\Gamma =  \pairing{\mu^{\star}}{f} - n.\]
  As a result, it suffices to minimise~$\loss$ on the set $\setc*{\mu \in \MM_+(K)}{\scl{\mu}{f} =  n}$.
 Any measure $\mu$ in the previous set satisfies $\mu(K) \leq 1/c$ thanks to the lower bound on the function~$f$.
 Thus, the set is bounded and hence weak-$\ast$ compact by the Banach--Alaoglu theorem~\cite{Rudin1991}.
 Since $\loss$ is clearly weak-$\ast$ continuous, the claim follows.
\end{proof}

\section{ML-EM algorithm}
\label{sec3}

In this section, we define the ML-EM algorithm, and prove that it is monotone and convergent in the low noise regime. Since all measures and continuous functions will from now on all be defined on the compact $K$ in this section, we drop the reference to $K$ in the functional spaces, denoting them $\MM$, $\MM_+$ and $\Cont$ respectively.
\subsection{Definition and well-posedness}
For $\mu_0 \in \operatorname{dom}(\loss)$, we define the iterates
\begin{equation}
\label{mlem}
\mu_{k+1} =\frac{\mu_k}{ f} \sum_{\gamma\in\Gamma} \frac{\gamma}{ \scl{\mu_k}{\gamma}}.
\end{equation}
We first observe that the algorithm is well-defined.
This is because
\[ \mu_k \in \operatorname{dom}(\loss) \Longrightarrow \mu_{k+1} \in \operatorname{dom}(\loss). \]
Indeed, using the lower bound~\eqref{Positive} on $f$, we find for any \(\gamma\in\Gamma\):
\begin{align}
\scl{\mu_{k+1}}{\gamma} & \geq \frac{1}{c\scl{\mu_{k}}{\gamma}} \scl{\mu_{k}}{\gamma^2}> 0, 
\end{align}
since the Cauchy--Schwarz inequality $\scl{\mu_{k}}{\gamma^2} \mu_k(K) \geq \scl{\mu_{k}}{\gamma}^2$ prevents $\scl{\mu_{k}}{\gamma^2}$ from vanishing.
More precisely,
after defining the compact
\[ \widetilde K:= \mathop{\bigcup_{\gamma\in\Gamma}} \operatorname{supp}(\gamma)
  ,
\]
we find that $\operatorname{supp}(\mu_k) = \widetilde K$ for all $k \geq 1$, provided that $\operatorname{supp}(\mu_0) = K$ and $\mu_0 \in \operatorname{dom}(\loss)$. Note that the optimality conditions~\eqref{kkt} prove that any optimal measure $\mu^\star$ satisfies $\operatorname{supp}(\mu^\star)  \subset \widetilde K$.

Finally, we remark that
\begin{equation}
\label{eq:constantmass}  
  \forall k \geq 1, \qquad \scl{\mu_k}{f} = n 
  .
\end{equation}

\begin{remark}
  Assume that for some times $t_0 < t_1, \ldots < t_N$ we have
  \[\forall t \in (t_{s-1}, t_s), \quad \mathcal{W}_t = \mathcal{W}_{t_s}, \quad s = 1, \ldots N,\] {i.e.}, the movement is piecewise constant on $(t_{s-1}, t_{s})$ for $s \in \set{1, \ldots N}$.
  We may then also group points by phase and detector by denoting $n_i^s$ the number of points detected in detector $i$ between $t_{s-1}$ and $t_s$, and $n^s = (n_1^s, \ldots, n_m^s)$. Using the notations $A_s := A \mathcal{W}_{t_s}$, $\Delta t_s = t_s-t_{s-1}$, the algorithm then rewrites 
\begin{align}
  \label{eq:piecewiseconstant}
\mu_{k+1} =\frac{\mu_k}{ \sum_{s=1}^N (\Delta t_s) A_s^\ast \boldsymbol{1}} \sum_{s=1}^{N} A_s^\ast\Big(\frac{n^s}{ A_s \mu_k}\Big).
\end{align}
In other words, we recover the algorithm for \emph{gated} data, proposed in~\cite{qiao06_motion_incor_recon_method_gated_pet_studies, Hinkle2012} for the intensity-preserving action, and generalised in~\cite{Li2019, Oktem2019}. 
The ensuing analysis is up to our knowledge the first rigorous justification for these informally-derived algorithms, under the assumption of piecewise-constant movement. 

Note also that \eqref{eq:piecewiseconstant} can be rewritten as
\begin{align}
  \label{eq:piecewiseconstant_}
  \mu_{k+1} =\frac{\mu_k}{ \sum_{s=1}^N  \tilde{A}_s^\ast \boldsymbol{1}} \sum_{s=1}^{N}  \tilde{A}_s^\ast\Big(\frac{n^s}{ \tilde{A}_s \mu_k}\Big).
\end{align}
where \(\tilde{A}_s \coloneqq (\Delta t_s) A_s \).
In this case, this is simply the standard ML-EM algorithm with the new operator \(\tilde{A} : \MM \to \RR^{\ndet N}\) defined by \(\tilde{A} \coloneqq \bracket{\tilde{A}_1,\dotsc,\tilde{A}_N}\).

In particular, we also recover the classical ML-EM algorithm, since if $\mathcal{W}_t = \mathrm{Id}$ for all $t\in[0,1]$, the above simplifies to 
\begin{align}
\mu_{k+1} =\frac{\mu_k}{ A^\ast \boldsymbol{1}} A^\ast\Big(\frac{y}{ A \mu_k}\Big),
\end{align}
with $y = (n_1, \ldots, n_\ndet)$.
\end{remark}

\subsection{Monotonicity.}

For a nonnegative function \(\gamma \in \mathcal{C}\) and a nonnegative measure \(\mu \in \MM_+\), we define the measure
\[
  \nu_{\gamma}(\mu) \coloneqq \frac{\gamma \mu}{\pairing{\mu}{\gamma}}
  .
\]
Note that \(\nu_{\gamma}\) is a probability measure over $K$, and that, following \eqref{mlem},
\begin{equation}
  \label{eq:mlemnu}
  \mu_{k+1} ={\frac{1}{f}} \sum_{\gamma\in\Gamma}{\nu_{\gamma}(\mu_k)}
  .
\end{equation}
Define the set $X_k:= \setc[\big]{\mu \in \MM_+}{\mu_{k+1} \ll \mu \ll \mu_k, \; \scl{\mu}{f} = n}$ for \(k\in\NN\). The conservation property~\eqref{eq:constantmass} shows that $\mu_k, \mu_{k+1} \in X_k$.

We now define the surrogate function \(Q_k\colon X_k \to \RR\) by
\[
  Q_k(\mu) \coloneqq \loss(\mu) + \sum_{\gamma\in\Gamma}  D\paren[\big]{\nu_{\gamma}(\mu_k) || \nu_{\gamma}(\mu)}
  ,
\]
where $D$ is the Kullback-Leibler divergence defined for any two probability measures $\mu \ll \nu$ over $K$ by \[D(\mu|| \nu) :=  \int_K\log\bigg(\frac{\dd \mu}{\dd \nu}\bigg) \, d\mu, \]with  $\frac{\dd \mu}{\dd \nu}$ standing for the Radon--Nikodym derivative of $\mu$ with respect to $\nu$.

\begin{lemma}
  \label{prop:surrogate}
The following holds for all $k \geq 1$:
\begin{enumerate}[label=\upshape(\roman*)]
  \item\label{surr:it:major} \(Q_k(\mu) \geq \loss(\mu), \qquad \mu \in X_k\)
\item\label{surr:it:contact} $Q_k (\mu_k) = \loss(\mu_k)$
\item\label{surr:it:div}
  \(Q_k(\mu) - Q_k(\mu_{k+1}) = \pairing{\mu - \mu_{k+1}}{f} + D(f \mu_{k+1}|| f \mu), \qquad \mu \in X_k \)
  \item\label{surr:it:surrit}
    \(Q_k(\mu_k) - Q_k(\mu_{k+1}) =  D(f \mu_{k+1}|| f \mu_k)\)

\end{enumerate}
\end{lemma}

\begin{proof}
  The properties of divergences
  allows us to conclude about \ref{surr:it:major} and \ref{surr:it:contact}.

  One computes
  \[Q_k(\mu) = \scl{\mu}{f} - \sum_{\gamma\in\Gamma} \pairing[\Big]{\nu_{\gamma}(\mu_k)}{\log\paren[\Big]{  \scl{\mu_k}{\gamma} \frac{d \mu}{d \mu_k} }}
    .\]
This gives
\begin{align*}
  Q_k(\mu) - Q_k(\mu_{k+1}) & = \\
  \pairing{\mu - \mu_{k+1}}{f}
                              &+\sum_{\gamma\in\Gamma}  \pairing[\bigg]{\nu_{\gamma}\paren{\mu_k}}{\log\paren[\Big]{\pairing{\mu_k}{\gamma} \frac{\dd\mu_{k+1}}{\dd\mu_k}} - \log\paren[\Big]{ \pairing{\mu_k}{\gamma} \frac{\dd\mu}{\dd\mu_{k}}}}     \\
                            & =
                              \pairing{\mu - \mu_{k+1}}{f}
                              +
                              \pairing[\bigg]{\underbrace{\sum_{\gamma\in\Gamma}\nu_{\gamma}\paren{\mu_k}}_{f \mu_{k+1}}}{\log\paren[\Big]{ \frac{\dd\mu_{k+1}}{\dd\mu}} }     \\
                            & =
                              \pairing{\mu - \mu_{k+1}}{f} +
                              D\paren[\big]{f\mu_{k+1}|| f\mu}.
\end{align*}
which proves \ref{surr:it:div}.
Finally, \ref{surr:it:surrit} is a consequence of \ref{surr:it:div} and \eqref{eq:constantmass}.
\end{proof}
These computations yield the monotony of the function $\loss$ along iterates.
\begin{corollary}
  \label{prop:lossdecrease}
  For any $\mu_0 \in \operatorname{dom}(\loss)$, we have
  \[ 0 \leq Q_k(\mu_k) - Q_k(\mu_{k+1}) \leq \loss(\mu_k) - \loss(\mu_{k+1})
    \qquad
    k \in \NN
    .\]
\end{corollary}
\begin{proof}
  
It is a consequence of \autoref{prop:surrogate}.
First, using \autoref{prop:surrogate}-\ref{surr:it:surrit}, we get \(Q_k(\mu_k) - Q_{k}(\mu_{k+1})   \geq 0\).
Now, conclude noticing that \autoref{prop:surrogate}-\ref{surr:it:major} and \autoref{prop:surrogate}-\ref{surr:it:contact} imply \(Q_k(\mu_k) - Q_k(\mu_{k+1}) = \loss(\mu_k) - Q_k(\mu_{k+1}) \leq \loss(\mu_k) - \loss(\mu_{k+1})\).
\end{proof}

\subsection{Convergence}
In this section, we highlight the main ideas of proofs, which largely follow~\cite{Verdier2020}.

\begin{proposition}
\label{fixed}
For any $\mu_0 \in \operatorname{dom}(\loss)$, the weak-$\ast$ cluster points $\bar \mu$ of $\set{\mu_k}_{k\in\NN}$ exist and are fixed points of the algorithm, namely
\begin{equation}
\bar \mu=\frac{\bar \mu}{ f} \sum_{\gamma\in\Gamma}\frac{\gamma}{ \scl{\bar \mu}{\gamma}}.
\end{equation}
\end{proposition}

\begin{proof}
  We first prove that the sequence is weak-$\ast$ compact in $\MM_+$.
  Integrating the defining relation of ML-EM \eqref{mlem}
 and using the assumption \eqref{Positive} that $f \geq c > 0$, we indeed find 
\begin{align}
\label{compact}
\mu_{k+1}(K) = \int_K \, \dd \mu_{k+1} \leq \frac{1}{c}  \sum_{\gamma\in\Gamma} \frac{\scl{\mu_k}{\gamma}}{\scl{\mu_k}{\gamma}}  = \frac{n}{c} 
  .
\end{align}
From the Banach--Alaoglu theorem~\cite{Rudin1991}, we extract a weak-$\ast$ converging subsequence to a given $\bar \mu$, and we denote the subsequence $\set{\mu_{\varphi(k)}}_{k\in\NN}$. 

We also note that $\loss$ is weak-$\ast$ continuous.
In particular, we have $\bar \mu \in \operatorname{dom}(\loss)$ since otherwise $\set{\loss(\mu_{\varphi(k)})}_{k\in\NN}$ would diverge to $+\infty$, in contradiction with its monotonicity. 

Using \autoref{prop:lossdecrease}
and \autoref{prop:surrogate}~\ref{surr:it:surrit}, we can then follow the lines of~\cite[Proposition 4.3]{Verdier2020} to conclude the proof.
\end{proof}

We refer to absolutely continuous measures for those that are absolutely continuous with respect to the Lebesgue measure on $K$, and proceed with a further assumption on the regularity of minimisers: 
\begin{equation}
\label{regularity}
\text{there exists $\mu^\star \in \argmin_{\mu \in \MM_+} \loss(\mu)$ absolutely continuous with $  \operatorname{supp}(\mu^\star)= \widetilde K$.}
\end{equation}

\begin{remark}
  Using the results in~\cite{Georgiou2005}, one typically expects the existence of such absolutely continuous measures in the low noise regime~\cite{Verdier2020}.
  More precisely, in the static case and if $y = (n_1, \ldots, n_\ndet)$ is in the interior of the cone $\setc{A \mu}{\mu \in \MM_+}$ in~$\mathbb{R}^\ndet$, such measures do exist~\cite{Verdier2020}.
\end{remark}

\begin{theorem}
Assume that assumption~\eqref{regularity} holds. Then, for any $\mu_0 \in \operatorname{dom}(\loss)$ absolutely continuous with a continuous and positive density over $K$, the algorithm is convergent in the sense
that \[\loss(\mu_k)\xrightarrow[k\rightarrow +\infty]{} \min_{\mu \in \MM_+} \loss(\mu).\]
Furthermore, any weak-$\ast$ limit point $\bar \mu$ of the algorithm satisfies $\operatorname{supp}(\bar \mu) = \widetilde K$.
\end{theorem}
\begin{proof}
We argue in several steps, letting $\bar \mu$ be a cluster point of the ML-EM iterates~$\set{\mu_k}_{k \in \NN}$.

\begin{enumerate}
  \item
The assumptions on $\mu_0$ ensure that $\mu_1$ is absolutely continuous, with  $\operatorname{supp}(\mu_1) = \widetilde K$, whence $D(f \mu^\star || f \mu_1) < +\infty$, where $\mu^\star$ is defined by~\eqref{regularity}. We can then prove similarly to~\cite[Proposition 4.7]{Verdier2020} that for all $k \geq 1$,
\[D(f \mu^\star || f \mu_{k+1}) \leq D(f \mu^\star || f \mu_{k}).\]
Taking a subsequence along which $\set{\mu_k}_{k \in \NN}$ converges weakly-$\ast$ to $\bar \mu$ and using the weak-$\ast$ lower semicontinuity of the Kullback--Leibler divergence~\cite{Posner1975}, we find \[D(f \mu^\star || f \bar \mu) < \infty.\] This shows that $f \mu^\star \ll f \bar \mu$, whence $\operatorname{supp}(\bar \mu) = \widetilde K$ thanks to the positivity of~$f$.

\item
We now make us of the fact that $\bar \mu$ must also be a fixed point of the algorithm, by virtue of~\autoref{fixed}. 
In other words, we have 
\[\bar \mu=\frac{\bar \mu}{ f} \sum_{\gamma\in\Gamma} \frac{\gamma}{ \scl{\bar \mu}{\gamma}} \Longleftrightarrow \bar \mu \, \nabla \loss(\bar \mu) = 0.\]
This implies 
$\nabla \loss(\bar \mu) = 0$ on $\operatorname{supp}(\bar \mu) $.
Outside of $\operatorname{supp}(\bar \mu) = \widetilde K$,
\[
  \nabla \loss(\bar \mu) = f - \sum_{\gamma\in\Gamma}\frac{\gamma}{ \scl{\bar \mu}{\gamma}} \geq 0
  ,
\]
since the right-hand side vanishes by definition of $\widetilde K$, whereas $f$ is positive.  Hence, $\bar \mu$ satisfies both optimality conditions~\eqref{kkt}, which shows that $\bar \mu$ is optimal and hence the claim that any cluster point is optimal.

\item
For the convergence of $\set{\loss(\mu_k)}_{k \in \NN}$ towards the minimum, we just recall that the sequence is non-increasing from~\autoref{prop:lossdecrease}, hence its limit must coincide with $\loss(\bar \mu)$ for any cluster point $\bar \mu$. The optimality of such cluster points concludes the proof.
\end{enumerate}

\end{proof}
The interest of emphasising the property $\operatorname{supp}(\bar \mu) = \widetilde K$ for cluster points is that this prevents them from being sparse measures.

\section{Numerical simulations}
\label{sec4}

All simulations are run in Python and use the Operator Discretization Library (\texttt{odl}) for manipulating operators~\cite{Adler2017}, \texttt{neuron} for warping utilities~\cite{Dalca2018}, which itself uses \texttt{tensorflow}~\cite{Abadi2016}.

\begin{algorithm}
  \caption{Pseudo-code for the ML-EM algorithm~\eqref{mlem}.}
  \label{alg:mainalgorithm}
  \begin{algorithmic}
    \State Compute $f :=  \int_0^1 \mathcal{W}_t^* A^* \one\, \dd t$.
    \State Choose initial $\mu_0$ and the number of iterates $k^\star$.
\Require $\mathbf{t} = (t_{i}^j)_{j = 1, \ldots, n_i, \, i = 1, \ldots, \ndet}$     \Comment{Times of detection}
    \For{$k \gets 0,\dotsc k^\star-1$}
         \State $a \gets 0$
    \For {$t_i^j \in \mathbf{t}$}
     \State $\gamma \gets \mathcal W_{t_i^j}^\ast a_i $ \Comment{These values could be stored instead}
     \State $a \gets a+ \frac{\gamma}{\langle \mu_k, \gamma \rangle}$ 
   \EndFor \Comment{Computes $\sum_{\gamma \in \Gamma} \frac{\gamma}{\langle \mu_k, \gamma \rangle}$}
    \State $\mu_{k+1}   \gets \frac{\mu_k}{f} a$
    \EndFor
    \State \textbf{return} $\mu_{k^\star}$
  \end{algorithmic}
\end{algorithm}


\subsection{General approach}
We work with the Derenzo phantom, displayed in the Introduction, see~\autoref{translation_evolution}. The phantom is explicitly defined by a function (hence, in continuum) and is subsequently discretised. 

The noise level is controlled by the dose (or time) by which we multiply the phantom.
In all experiments we have run, we have multiplied the Derenzo phantom by $10$. 
We denote it~$\mu_r$.

Given a time-evolution of the template through operators $\mathcal W_t$, we first generate the number of points and times per detector, using the result established in \autoref{Inhomogeneous}.
The latter states that they are for each $i \in \set{1, \ldots ,\ndet}$ independently drawn according to an inhomogeneous Poisson point process over $[0,1]$, with intensity defined by
\[
  t \in [0,1] \mapsto \beta_i(t) = \int_K a_i \, \dd (\mathcal{W}_t \mu_r) = \int_K (\mathcal{W}_t^\ast a_i) \,\dd \mu_r
  .
\]
 These processes are simulated using the rejection method, which goes as follows 
 \begin{itemize}
\item find a bound $M_i$ such that $\beta_i \leq M_i$ on $[0,1]$,
\item simulate a homogeneous Poisson process of intensity $M_i$ on $[0,1]$, 
\item accept a drawn time $t_j^i$ with probability $\beta_i(t_j^i)/M_i$.
\end{itemize}

 In each of the test cases presented below, we also compute the total number of points detected by detector $n_i$, $ i \in \set{1, \dotsc ,\ndet}$, thus obtaining a sinogram, which we call \emph{aggregated data}.
 Data can be aggregated on the whole interval of time $[0,1]$ or on a small portion of it to curb the effect of movement.
 This is what is commonly done in modern scanners to alleviate blur coming from movement of organs, such as the lungs. 

Finally, we will also compute a fictitious sinogram data, as obtained from the static phantom acquired for the same amount of time.

All in all, this gives us three benchmarks again which we may test our results for a given fixed number of iterates $k^\star$: 
\begin{itemize}
\item $k^\star$ iterates of the classical ML-EM algorithm on data acquired from the static phantom.
\item $k^\star$ iterates of the classical ML-EM algorithm on the aggregated data on the whole interval $[0,1]$.
\item $k^\star$ iterates of the classical ML-EM algorithm on the aggregated data on a relevant subinterval of $[0,1]$.
\end{itemize}

We expect our algorithm to perform as well as the classical ML-EM algorithm on a static phantom, thus avoiding both the blur observed when aggregating data on the whole time-interval because of movement, and the higher noise observed when only a portion of the aggregated data is kept.


In both cases, we work with a 2D PET operator $A$ with $45$ angles (views) and $64$ tangential positions, and the image space is the square $[-20,20]^2$ with resolution $128 \times 128$.
\subsection{Translation}
\label{sec:translation}


The image $\mu_t$ is given as the evolution of $\mu_r$ through operators $\mathcal W_t$ defined by means of translations.
We choose a mapping $c : [0,1] \mapsto \RR$ for a speed of translation, the operators are correspondingly defined through their adjoint for functions $f$ by
\[x \mapsto \mathcal W_t^\ast f(x) =f(x+c(t)). \]
In the experiment 
 the image is translated from left to right at speed $t$, up until it is at the center of the image, at $t =\frac{3}{4}$. The translation then stops.
 In other words, \[c(t) = (c_1(t), 0), \quad c_1(t) = (a t +b) \mathbbm{1}_{[0,\frac{3}{4}]} (t), \quad t \in [0,1]\] for appropriately chosen constants $a$ and $b$.
The resulting evolving image is depicted in \autoref{translation_evolution}.

\autoref{translation_comp} shows the result of $10$ iterates of our algorithm on the times $t_i^j$, compared to $10$ iterates of the classical ML-EM algorithm as obtained either on a fictitious static case, or on partially or fully aggregated data. The partially aggregated data is taken from the time-inverval $[\frac{3}{4}, 1]$, namely when the translation has stopped.

As expected, the result is almost undistinguishable from the static case, whereas the classical ML-EM algorithm on the full aggregated data leads to a poor image due to the movement.
The same applied to partially aggregated data performs well since the image is static on the last portion of the time-window, but $3/4$th of the data is unused, resulting in a noisier image.
The small circles towards the center are indeed less easily distinguished in the case \subref{trans:partial} as they can be in cases \subref{trans:proposed} and \subref{trans:classical}.

We observe a lower maximum for image \subref{trans:aggregated} than for images \subref{trans:proposed} and \subref{trans:classical}.
This is to be expected, as classical ML-EM iterates are such that $\scl{\mu_k}{A^\ast \one} =n$, whereas it is $\scl{\mu_k}{f} = n$ (see \eqref{eq:constantmass}) for our algorithm.
Since the function $f$ satisfies $f \leq A^\ast \one$ in this translating situation, maxima must indeed differ.

Finally, we stress that we have scaled the last image \subref{trans:partial} by a multiplication of~$4$ to make the comparison with \subref{trans:proposed} and \subref{trans:classical} easier, since there is only $1/4$th of the data.

\subsection{Mass-preserving action of diffemorphisms}


The set of images $(\mu_t)_{0 \leq t \leq 1}$ is given as the evolution of $\mu_r$ through operators $\mathcal W_t$ defined by means of diffeomorphisms $\varphi_t$ and the mass-preserving action, {i.e.}, for all $f\in \mathcal{C}$, 
\begin{equation}
\label{mass_preserving}
\mathcal W_t^\ast f(x) = f(\varphi_t(x)), \; x \in K.
\end{equation}
Now, if $\mu$ is absolutely continuous with density $g$ with respect to the Lebesgue measure, we find
\[\scl{\mathcal{W} \mu}{f} = \int_E f(\varphi(x)) g(x) \, \dd x =  \int_E  f(x)  |D\varphi^{-1}(x)| g(\varphi^{-1}(x)) \, \dd x,\]
which shows that $\mathcal{W} \mu$ is absolutely continuous with density $x \mapsto |D\varphi^{-1}(x)| g(\varphi^{-1}(x))$ with respect to the Lebesgue measure.

 The diffeomorphisms are obtained by integration of a stationary vector field, as $\varphi_t = \exp(t  v)$. In other words, they are obtained by integrating the following Cauchy problem over $[0,1]$: 
\begin{equation}
\label{exp}
  \begin{cases}
    \partial_t \varphi_t(x)=v(\varphi_t(x)),\\
    \varphi_0 =\mathrm{Id}.
  \end{cases}
\end{equation}
The resulting evolving image is depicted in \autoref{diffeo_evolution}.

\begin{figure}
\begin{subfigure}{0.32\textwidth}
\includegraphics[width=\linewidth]{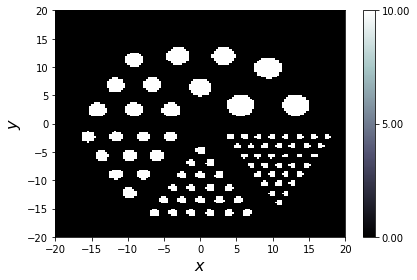} 
\caption{$t = 0$} 
\end{subfigure} 
\hfill
\begin{subfigure}{0.32\textwidth}
\includegraphics[width=\linewidth]{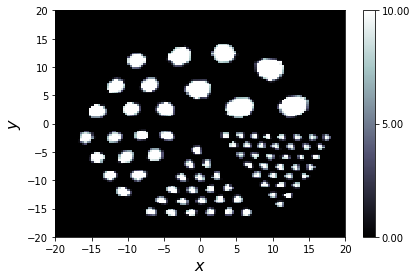} 
\caption{$t = 0.2$} 
\end{subfigure} 
\hfill
\begin{subfigure}{0.32\textwidth}
\includegraphics[width=\linewidth]{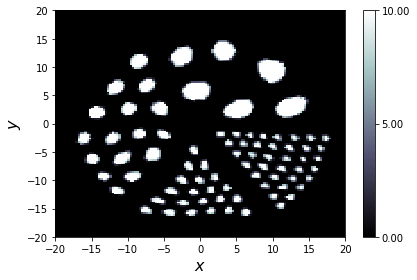} 
\caption{$t = 0.4$} 
\end{subfigure}

\bigskip

\begin{subfigure}{0.32\textwidth}
\includegraphics[width=\linewidth]{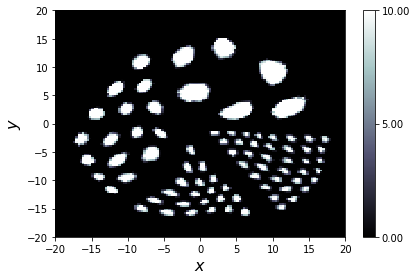} 
\caption{$t = 0.6$} 
\end{subfigure} 
\hfill
\begin{subfigure}{0.32\textwidth}
\includegraphics[width=\linewidth]{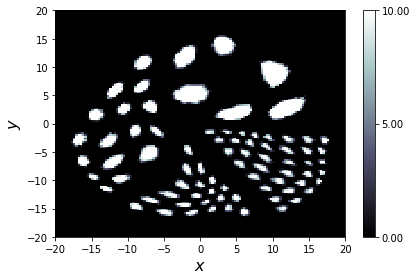} 
\caption{$t = 0.8$}
\end{subfigure} 
\hfill
\begin{subfigure}{0.32\textwidth}
\includegraphics[width=\linewidth]{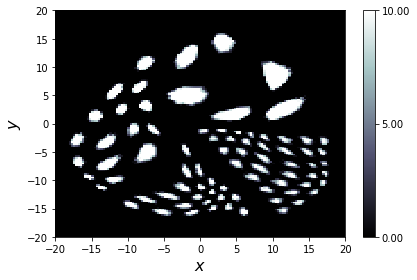} 
\caption{$t = 1$} 
\end{subfigure} 
\caption{Evolution of a template with mass-preserving operators.} \label{diffeo_evolution}
\end{figure}

\autoref{diffeo_comp} presents the results after ten iterations of the proposed algorithm.
 Again, the results of the classical ML-EM algorithm on the static phantom after the same number of iterations look very similar.

However, the reconstruction obtained on the full aggregated data is unsurprisingly blurred, as all the positions are ``averaged'' in the reconstruction.
The last reconstruction obtained from aggregating data on $[0,\frac{1}{4}]$ exhibits less blur but more noise.

\begin{figure}

\begin{subfigure}{0.49\textwidth}
\includegraphics[width=\linewidth]{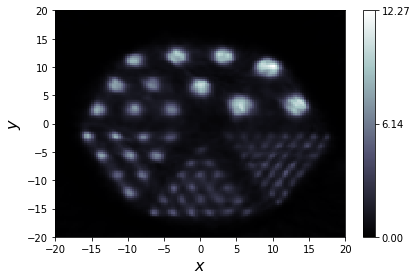} 
\caption{Proposed ML-EM algorithm} 
\end{subfigure} 
\hfill
\begin{subfigure}{0.49\textwidth}
\includegraphics[width=\linewidth]{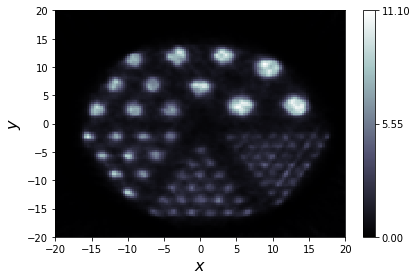} 
\caption{Classical ML-EM on static data} 
\end{subfigure} 

\bigskip
\begin{subfigure}{0.49\textwidth}
\includegraphics[width=\linewidth]{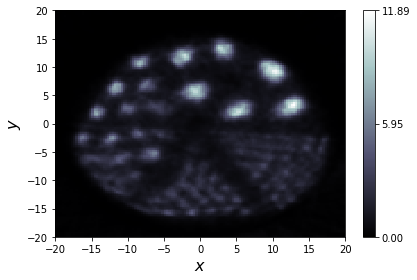} 
\caption{Classical ML-EM on full aggregated data \quad} \label{fig:comp:agg}
\end{subfigure} 
\hfill
\begin{subfigure}{0.49\textwidth}
\includegraphics[width=\linewidth]{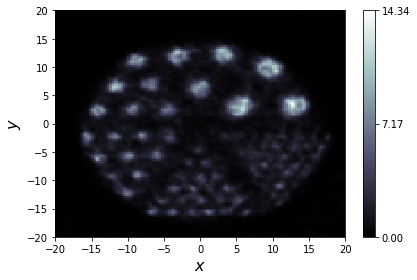} 
\caption{Classical ML-EM on partial aggregated data on $[0,\frac{1}{4}]$, scaled} 
\end{subfigure} 
\caption{$10$th iterate of the proposed algorithm \eqref{mlem} compared to the $10$th iterate of the classical ML-EM algorithm in various cases.} \label{diffeo_comp}
\label{fig:comparison}
\end{figure}

\subsection{Investigating the effect of a wrong motion model}

Finally, we provide an experiment mimicking the case where the available motion model would be slightly wrong, such as what would happen if there were measurement errors or if the motion model were estimated jointly with the image. We work with the previous example and encode the error in the noise model through a wrong vector field. 

In other words, we are given a wrong vector field $v_\delta$ with $\|{v-v_\delta}\|_\infty$ small. This vector field generates wrong diffeomorphisms $\varphi_\delta$ defined by~\eqref{exp} with $v_\delta$ instead of $v$. These in turn generate a wrong motion model with operators $(W_\delta)_t$ by the mass-preserving action~\eqref{mass_preserving}.

In~\autoref{wrong_diffeo}, the template both with the correct vector field $v$ and with the wrong vector field $v_\delta$ is depicted at time $t=1$.
We then compare in the same Figure the effect of $10$ ML-EM iterates~\eqref{mlem} both with the correct motion model and the wrong one, together with $10$ classical ML-EM iterates on the full aggregated data.

The effect of the wrong vector field is clearly visible but the main features of the Derenzo phantom may still be identified. As a result, this stands as numerical evidence of a stability property of ML-EM iterates~\eqref{mlem} with respect to an error in the motion model, at least when it is encoded through the vector field defining diffemorphisms underlying the mass preserving motion model.

A visual comparison with the results presented in~\autoref{diffeo_comp} also shows that one might prefer keeping only a portion of the data corresponding to a phase with negligible movement if the movement model is known with too much uncertainty. Hence, the ML-EM iterates~\eqref{mlem} yield an actual improvement if either the movement model is known with enough precision or if one has to restrict to a very small portion of data to neglect movement, leading to too noisy results.

\section{Discussion and perspectives}
\label{sec5}

\subsection{Possible generalisations}

We discuss how the statistical framework we have developed can be extended to more realistic physical models for emission tomography. 

One extension is to take into account the attenuation map, usually given by a CT-scan.
We assume that the CT-scan has been aligned with the PET scan at a given fixed time.
We thus have to replace the operator \(A\) by a time-dependent one.
It means that the functions $a_i$ time-dependent, that is, $a_i(t)$.
This in turns defines the time-dependent operator $A(t)$ acting on measures by \(A_i(t) \mu:= \int_K a_i(t) \, d\mu\).

Another improvement is time-of-flight PET, which takes advantage of the detection time differences for two opposite detectors.
The system matrix in a continuous time setting is certainly more involved, but it should be possible to use the standard time-of-flight operators consisting of fixed portions of lines (``time of flight bins'') instead.
Since we make no particular assumption on the operator \(A\), our algorithm can then be used without further modification.

\subsection{Relations to ML-EM algorithms for gated data}
As explained both for the maximum likelihood problem and the ML-EM algorithm, our approach reduces to ML-EM in the gated data case.
An interesting question is to consider if and how our algorithm can then be rigorously derived. 

A first formal approach would be to consider that there are as many gates as there are detection times, in which case the ML-EM algorithm for gated data~\eqref{eq:piecewiseconstant} is very close to our algorithm~\eqref{mlem}.
However, the denominator would then be wrong: it is a (random) Riemann sum approximation of the function $f$.
Moreover, the precise relation between the two requires further analysis as the waiting times between two times of detection are exponentially-distributed random variables.

Another (a posteriori) approach is to consider $\Delta t$ small enough so that each time-interval of size $\Delta t$ as at most one detection point.
Then again, the algorithm~\eqref{eq:piecewiseconstant} does not exactly lead to~\eqref{mlem} because the denominator \(f\) is off due to the time discretisation.

Note however that letting \(\Delta t\) go to zero is certainly possible to obtain our algorithm, but ripe with technical difficulties: in which sense is the convergence, at which \(\Delta t\) converges to zero compared to data acquisition, etc.
Our approach avoids these difficulties by providing instead a comprehensive modelling framework.
This model directly considers movement as continuous in time rather than passing to the limit from piecewise-constant movement models.

\subsection{Numerical implementation}
One reason behind the popularity of the ML-EM algorithm (or its more efficient counterpart OSEM)
in modern PET scanners 
is its computational tractability.
Hence, the effort made in the present work to derive an algorithm with comparable computational burden requires that the computation of motion operators be at most of the order of that of the PET operator.

As a close inspection of our algorithm shows, $f$ should be computed offline, just as $A^T 1$ is for the classical ML-EM iterate.
Given the times $t_i^j$, we note that the elements of $\Gamma$ (namely $\mathcal{W}_{t}^\ast a_i$ for all times $t_i^j$, $i = 1, \ldots, \ndet, \; j  = 1, \ldots, n_i$) can also be computed  offline. Once this is done, the iterates do not require any further use of the PET operator $A$. 
This is of course possible only with few detection points, since each element of $\Gamma$ has the size of one image.

We also emphasise that working with list-mode data as we do allows to start reconstructing even with very few time points, adding newly detected points along further iterates of the algorithm. 

\subsection{From continuum to discrete}
The framework and algorithm we propose are continuous, but the resulting algorithm is obviously used on discrete images, using a corresponding discrete operator. 
The seminal paper~\cite{Shepp1982} starts from continuous images (measures for the Poisson processes) and directly discretises without explicitly describing the discrepancies between the discrete and the continuous model, nor how they translate at the level of the ML-EM algorithm.

To the best of our knowledge, the discretisation error for PET reconstruction, even for the common ML-EM algorithm, has never been investigated.
For ill-posed problems, it is well-known that discretisation can have regularising properties; one could expect to observe the same behaviour and we believe any attempt at quantifying this effect would be valuable.

\subsection{Joint estimation of the motion model and the template}
As briefly mentioned in the introduction, one natural extension is the unknown movement model case.


Let us assume that the movement is parameterised by some element $v$ of a suitably chosen space $X$ (such as when the motion operators are defined through diffemorphisms themselves obtained by integration of a vector field).
We obtain an optimisation problem of the form
\[ \min_{\mu \in \MM_+(K), \, v \in X} \qquad \loss(\mu, v) + \alpha \mathcal{R}(v),\]
with $l(\mu,v) = \pairing{\mu}{f(v)} - \sum_{\gamma\in\Gamma(v)} \log(\scl{\mu}{\gamma})$: the dependence in $v$ is in $f$ and $\Gamma$.

The second term $\mathcal{R}(v)$ is a regularisation term acting on the movement model parameter, and $\alpha>0$ is a regularisation parameter.
Note in this case we incorporate the motion model as a hard constraint since the optimisation problem is posed directly on the template at time zero.

Solving the optimisation problem is usually done by alternating minimisation, namely by iteratively optimising with respect to $\mu \in \MM_+(K)$ and $v \in X$, leading to an outer loop with sequences of estimates $(\mu_p, v_p)_{p \in \mathbb{N}}$. 
Given a current estimate $v_p$ for $v \in X$, the next estimate for $\mu_{p+1}$ is obtained by solving $\min_{\mu \in \MM_+(K)} \loss(\mu, v_p)$, which can precisely be done by our algorithm through an inner loop.

In this direction, the article~\cite{Burger2018} handles a large class of functionals aimed at joint motion estimation and image reconstruction. Unfortunately, the functional we are dealing with does not fit in this framework. Although some of the mathematical techniques presented in the latter reference will prove useful, efficiently solving the joint problem in our case requires ad hoc methods and is the subject of a future work. 


\begin{figure}
 \captionsetup{size=small}

\begin{subfigure}{0.49\textwidth}
\includegraphics[width=\linewidth]{Diffeo_t1.png} 
\caption{Template at $t=1$, true motion model} 
\end{subfigure} 
\begin{subfigure}{0.49\textwidth}
\includegraphics[width=\linewidth]{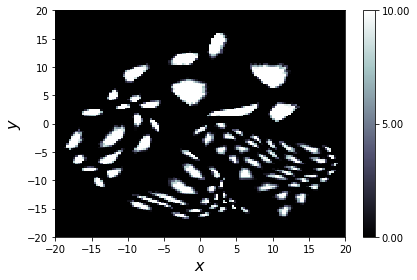} 
\caption{Template at $t=1$, wrong motion model}
\end{subfigure} 
\begin{subfigure}{0.49\textwidth}
\includegraphics[width=\linewidth]{Diffeo_reconst_mlem_agg_10.png} 
\caption{Classical ML-EM on full aggregated data \quad} \label{fig:wrong:agg}
\end{subfigure} 
\begin{subfigure}{0.49\textwidth}
  \includegraphics[width=\linewidth]{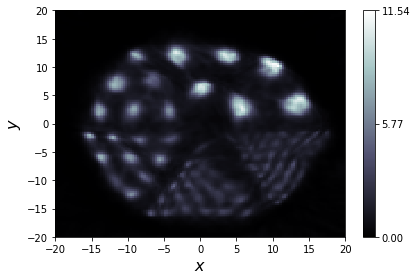} 
  \caption{Proposed ML-EM algorithm, wrong motion model \quad} 
\end{subfigure} 
\begin{subfigure}{0.49\textwidth}
  \includegraphics[width=\linewidth]{Diffeo_reconst_mlem_mov_10.png} 
  \caption{Proposed ML-EM algorithm, true motion model \quad} 
\end{subfigure} 

\caption{Effect of a wrongly-estimated movement model after $10$ iterates of~\eqref{mlem}.
  \subref{fig:wrong:agg} is a duplicate of \autoref{fig:comp:agg} for convenient comparison.
}
\label{wrong_diffeo}

\end{figure}

\bibliography{biblio.bib}
\bibliographystyle{acm}
\end{document}